\documentclass[12pt]{article}
\usepackage{amssymb,amsthm,amsmath,latexsym}
\usepackage{url}
\usepackage{lscape} 
\newtheorem{thm}{Theorem}
\newtheorem{prop}[thm]{Proposition}

\theoremstyle{remark}
\newtheorem{rem}[thm]{Remark}
\newcommand{\FF}{\mathbb{F}}

\newcommand{\BB}{\mathcal{B}}
\newcommand{\CC}{\mathcal{C}}
\newcommand{\cD}{\mathcal{D}}
\newcommand{\0}{\mathbf{0}}
\newcommand{\1}{\mathbf{1}}
\DeclareMathOperator{\Aut}{Aut}
\DeclareMathOperator{\wt}{wt}

\begin{document}

\title{On the classification of linear complementary dual codes}

\author{
Makoto Araya\thanks{Department of Computer Science,
Shizuoka University,
Hamamatsu 432--8011, Japan.
email: {\tt araya@inf.shizuoka.ac.jp}}
and 
Masaaki Harada\thanks{
Research Center for Pure and Applied Mathematics,
Graduate School of Information Sciences,
Tohoku University, Sendai 980--8579, Japan.
email: {\tt mharada@m.tohoku.ac.jp}}
}

\maketitle

\begin{abstract}
We give a complete classification of binary
linear complementary dual codes of lengths up to $13$ and
ternary 
linear complementary dual codes of lengths up to $10$.
\end{abstract}

\section{Introduction}
Let $\FF_q$ denote the finite field of order $q$,
where $q$ is a prime power.
An $[n,k]$ {\em code} $C$ over $\FF_q$ 
is a $k$-dimensional vector subspace of $\FF_q^n$.
A code over $\FF_2$ is called {\em binary} and
a code over $\FF_3$ is called {\em ternary}.
The parameters $n$ and $k$
are called the {\em length} and {\em dimension} of $C$, respectively.
Two $[n,k]$ codes $C$ and $C'$ over $\FF_q$ are 
{\em equivalent}, denoted $C \cong C'$,
if there is an $n \times n$ monomial matrix $P$ over $\FF_q$ with 
$C' = C \cdot P = \{ x P \mid x \in C\}$.  

The {\em dual} code $C^{\perp}$ of an $[n,k]$ code $C$ over $\FF_q$
is defined as
$
C^{\perp}=
\{x \in \FF_q^n \mid x \cdot y = 0 \text{ for all } y \in C\},
$
where $x \cdot y$ is the standard inner product.
A code $C$ is called {\em linear complementary dual}
(or a linear code with complementary dual)
if $C \cap C^\perp = \{\0_n\}$, where $\0_n$ denotes the zero vector of length $n$.
We say that such a code is LCD for short.
LCD codes were introduced by Massey~\cite{Massey} and gave an optimum linear
coding solution for the two user binary adder channel.
Recently, much work has been done concerning LCD codes
for both theoretical and practical reasons (see~\cite{mf}, 
\cite{CMTQ2}, \cite{bound}, \cite{HS}
and the references therein).
In particular, we emphasize the recent work by 
Carlet, Mesnager, Tang, Qi and Pellikaan~\cite{CMTQ2}.
It has been shown in~\cite{CMTQ2} that
any code over $\FF_q$ is equivalent to some LCD code
for $q \ge 4$.
This motivates us to study binary LCD codes and ternary LCD codes.
In addition, 
recently, Carlet, Mesnager, Tang and Qi~\cite{mf} have established 
the mass formulas.
This motivates us to start a classification of binary LCD codes and 
ternary LCD codes.
The aim of this note is to
give a complete classification of binary
linear complementary dual codes of lengths up to $13$ and
ternary 
linear complementary dual codes of lengths up to $10$.

The note is organized as follows.
In Section~\ref{Sec:2},
definitions, notations and basic results are given.
We give some observation on the classification of
binary LCD codes $C$ and
ternary LCD codes $C$ with $d(C^\perp)=1$.
Section~\ref{Sec:2} also presents the mass formulas given in~\cite{mf}
for binary LCD codes and ternary LCD codes.
The mass formulas are an important role in the classification of 
binary LCD codes and ternary LCD codes.
In Section~\ref{Sec:sd}, 
we give a complete classification of 
binary LCD $[n,k]$ codes and 
ternary LCD $[n,k]$ codes for $k=1,n-1$.
In Section~\ref{Sec:B}, we give a complete classification of
binary LCD codes of lengths up to $13$.
In Section~\ref{Sec:T}, we give a complete classification of
ternary LCD codes of lengths up to $10$.

\section{Preliminaries}
\label{Sec:2}

\subsection{Definitions, notations and basic results}


Let $C$ be an $[n,k]$ code over $\FF_q$.  
The {\em weight} $\wt(x)$ of a vector $x \in \FF_q^n$ is
the number of non-zero components of $x$.
A vector of $C$ is called a {\em codeword} of $C$.
The minimum non-zero weight of all codewords in $C$ is called
the {\em minimum weight} $d(C)$ of $C$ and an $[n,k]$ code with minimum
weight $d$ is called an $[n,k,d]$ code.
The {\em weight enumerator} of $C$ is given by
$\sum_{i=0}^n A_i y^i$, where
$A_i$ is the number of codewords of weight $i$ in $C$.
An {\em automorphism} of $C$ is an $n \times n$ monomial matrix $P$ 
over $\FF_q$ with $C = C \cdot P$.  
The set consisting of all automorphisms of $C$ is called the
{\em automorphism group} of $C$ and it is denoted by $\Aut(C)$.
A {\em generator matrix} of $C$
is a $k \times n$ matrix whose rows are a set of basis vectors of $C$.
A {\em parity-check matrix} of $C$ is a generator matrix of $C^\perp$.

Throughout this note, we use the following notations.
Let $\0_{n}$ denote the zero vector of length $n$ and
let $\1_{n}$ denote the all-one vector of length $n$.
Let $I_n$ denote the identity matrix of order $n$ and
let $A^T$ denote the transpose of a matrix $A$.

The following characterization is due to Massey~\cite{Massey}.

\begin{prop}
Let $C$ be a code over $\FF_q$.  
Let $G$ and $H$ be a generator matrix and a parity-check
matrix of $C$, respectively.
Then the following properties are equivalent:
\begin{itemize}
\item[\rm (i)] $C$ is LCD,
\item[\rm (ii)] $C^\perp$ is LCD,
\item[\rm (iii)] $G G^T$ is nonsingular,
\item[\rm (iv)] $H H^T$ is nonsingular.
\end{itemize}
\end{prop}


The following proposition is trivial.

\begin{prop}
Suppose that $C$ is an LCD code over $\FF_q$ and $q \in \{2,3\}$.
If $C'$ is equivalent to $C$, then $C'$ is also LCD.
\end{prop}

Throughout this note, we use the following notations.
Let $\BB_{n,k}$ denote the set of all inequivalent
binary LCD $[n,k]$ codes.
Let $\BB_{n,k,d}$ denote the set of all inequivalent
binary LCD $[n,k,d]$ codes.
Let $\CC_{n,k}$ denote the set of all inequivalent
ternary LCD $[n,k]$ codes.
Let $\CC_{n,k,d}$ denote the set of all inequivalent
ternary LCD $[n,k,d]$ codes.

\subsection{LCD codes $C$ with $d(C^\perp)=1$}
\label{subsec:dd1}

Let $C$ be an $[n,k,d]$ code over $\FF_q $.
Define the following $[n+1,k,d]$ code:
\begin{equation}\label{eq:extend}
C^* = \{(x_1,x_2,\ldots,x_{n},0) \mid (x_1,x_2,\ldots,x_{n}) \in C\}. 
\end{equation}
Let $D$ be an $[n+1,k,d]$ code over $\FF_q $ with $d(D^\perp)=1$.
It is easy to see that there is an $[n,k,d]$ code $C$ over $\FF_q$
with $D \cong C^*$.
It is trivial that
$C^*$ is LCD if and only if $C$ is LCD. 
%
%
In this way, every LCD $[n+1,k,d]$ code $D$ over $\FF_q$ 
with $d(D^\perp)=1$
is constructed from some LCD $[n,k,d]$ code $C$  over $\FF_q$.
In addition, two LCD $[n+1,k,d]$ codes $D$ over $\FF_q$ with $d(D^\perp)=1$
are equivalent if and only if 
two LCD $[n,k,d]$ codes $C$  over $\FF_q$ are equivalent.
Hence, all LCD $[n+1,k,d]$ codes $D$ over $\FF_q$ with $d(D^\perp)=1$,
which must be checked to achieve
a complete classification, can be obtained from 
all inequivalent LCD $[n,k,d]$ codes $C$  over $\FF_q$.
Therefore, we have the following:

\begin{prop}\label{prop:dd1}
Let $\BB^*_{n,k}$ denote the set of all inequivalent
binary LCD $[n,k]$ codes $B$ with $d(B^\perp)=1$.
Let $\CC^*_{n,k}$ denote the set of all inequivalent
ternary LCD $[n,k]$ codes $C$ with $d(C^\perp)=1$.
\begin{itemize}
\item[\rm (i)] 
There is a set $\BB_{n-1,k}$ of all inequivalent
binary LCD $[n-1,k]$ codes such that
$\BB^*_{n,k} =\{B^* \mid B \in \BB_{n-1,k}\}$.
\item[\rm (ii)] 
There is a set $\CC_{n-1,k}$ of all inequivalent
ternary LCD $[n-1,k]$ codes 
such that $\CC^*_{n,k} =\{C^* \mid C \in \CC_{n-1,k}\}$.
\end{itemize} 
\end{prop}

\subsection{Mass formulas}

It is trivial that $\{\0_n\}$ and $\FF_q^n$ are 
the unique LCD $[n,1]$ code over $\FF_q$ and
the unique LCD $[n,n]$ code over $\FF_q$, respectively.
From now on, we assume that 
\[
1 \le k \le n-1
\]
for an LCD  $[n,k]$ code over $\FF_q$.


Let $\mathcal{B}_{n,k}$ denote the set of all inequivalent
binary LCD $[n,k]$ codes.
Using the Gaussian binomial coefficients, 
the following values
\begin{equation}\label{eq:mf}
T_{2}(n,k) = \sum_{B \in \mathcal{B}_{n,k}} \frac{n!}{|\Aut(B)|}
\end{equation}
were determined theoretically in~\cite[Corollary~4.8]{mf}, without
finding the set $\mathcal{B}_{n,k}$, as follows:
\[
T_{2}(n,k)=
\begin{cases}
2^{\frac{nk-k^2+n-1}{2}}
\left[\frac{n}{2}-1\atop \frac{k-1}{2}\right]_4
&\text { if $n$ is even and $k$ is odd,}
\\
2^{\frac{(n-k)(k+1)}{2}}
\left[\frac{n-1}{2} \atop \frac{k-1}{2}\right]_4
&\text { if $n$ and $k$ are odd,}
\\
2^{\frac{k(n-k+1)}{2}}
\left[\frac{n-1}{2} \atop \frac{k}{2}\right]_4
&\text { if $n$ is odd and $k$ is even,}
\\
2^{\frac{k(n-k)}{2}}
\left(2^{n-k}
\left[\frac{n}{2}-1 \atop \frac{k}{2}-1\right]_4
+
\left[\frac{n}{2}-1 \atop \frac{k}{2}\right]_4
\right)
&\text { if $n$ and $k$ are even,}
\end{cases}
\]
where
\[
\left[n\atop k\right]_q
=\frac{(q^n-1)(q^{n-1}-1)\cdots(q^{n-k+1}-1)}
{(q-1)(q^2-1) \cdots (q^k-1)} \text{ if }k \ne 0,
\]
and $\left[n \atop 0\right]_q=1$.
The value $T_2(n,k)$ is the number of the distinct binary LCD
$[n,k]$ codes.

Let $\CC_{n,k}$ denote the set of all inequivalent
ternary LCD $[n,k]$ codes.
Similar to the above,
the following values
\begin{equation}\label{eq:mf3}
T_{3}(n,k) = \sum_{C \in \mathcal{C}_{n,k}} \frac{2^n n!}{|\Aut(C)|}
\end{equation}
were determined theoretically in~\cite[Corollary~5.9]{mf}, without
finding the set $\mathcal{C}_{n,k}$, as follows:
\[
T_{3}(n,k)=
\begin{cases}
3^{\frac{nk-k^2-1}{2}} 3^{\frac{n}{2}-1} 
\left[ \frac{n}{2}-1 \atop \frac{k-1}{2}  \right]_9
&\text { if $n \equiv 0 \pmod 4$ and $k$ is odd,}
\\
3^{\frac{nk-k^2-1}{2}} 3^{\frac{n}{2}+1} 
\left[ \frac{n}{2}-1\atop   \frac{k-1}{2}\right]_9
&\text { if $n \equiv 2 \pmod 4$ and $k$ is odd,}
\\
3^{\frac{(k+1)(n-k)}{2}}  
\left[ \frac{n-1}{2}  \atop \frac{k-1}{2} \right]_9
&\text { if $n$ is odd and $k$ is odd,}
\\
3^{\frac{k(n-k+1)}{2}} 
\left[  \frac{n-1}{2} \atop\frac{k}{2}  \right]_9
&\text { if $n$ is odd and $k$ is even,}
\\
3^{\frac{k(n-k)}{2}} 
\left[ \frac{n}{2}\atop  \frac{k}{2} \right]_9
&\text { if $n$ is even and $k$ is even.}
\end{cases}
\]
The value $T_3(n,k)$ is the number of the distinct ternary LCD
$[n,k]$ codes.
The equations~\eqref{eq:mf} and~\eqref{eq:mf3}
are called the {\em mass formulas}.

\section{Classification for dimensions $1$ and $n-1$}
\label{Sec:sd}

In this section,
we give a complete classification of 
binary LCD $[n,k]$ codes and 
ternary LCD $[n,k]$ codes for $k=1,n-1$.

\begin{prop}\label{prop:dim1B}
Let $\BB_{n,k}$ denote the set of all inequivalent
binary LCD $[n,k]$ codes.
Let $\BB_{n,k,d}$ denote the set of all inequivalent
binary LCD $[n,k,d]$ codes.
\begin{itemize}
\item[\rm (i)] 
$
|\BB_{n,1,d}|=
\begin{cases}
1 & \text{ if $n$ is even and $d=1,3,5,\ldots,n-1$,} \\
1 & \text{ if $n$ is odd and $d=1,3,5,\ldots,n$,}  \\
0 & \text{ otherwise.}
\end{cases}
$
\item[\rm (ii)] 
$
|\BB_{n,1}|=
|\BB_{n,n-1}|=
\begin{cases}
\frac{n}{2} & \text{ if $n$ is even,} \\
\frac{n+1}{2} & \text{ if $n$ is odd.} 
\end{cases}
$
\item[\rm (iii)] 
$
|\BB_{n,n-1,d}|=
\begin{cases}
\frac{n}{2} & \text{ if $n$ is even and $d=1$,} \\
\frac{n-1}{2} & \text{ if $n$ is odd and $d=1$,} \\
1 & \text{ if $n$ is odd and $d=2$,}  \\
0 & \text{ otherwise.}
\end{cases}
$
\end{itemize}
\end{prop}
\begin{proof}
By considering a permutation of the coordinates, 
we may assume without loss of generality that 
a binary $[n,1]$ code has generator matrix 
of the form
$\left(
\begin{array}{cc}
1 & a
\end{array}
\right)
$,
where $a=(a_1,a_2,\ldots,a_{n-1}) \in \FF_2^{n-1}$.
We denote the code by $B(a)$.
It is trivial that
$B_n(a)$ is LCD if and only if $\wt(a)$ is even.
If $\wt(a) = \wt(b)$, then it is trivial that
$B_n(a)$ and $B_n(b)$ are equivalent.
If $\wt(a) \ne \wt(b)$, then
$B_n(a)$ and $B_n(b)$ are inequivalent 
since $B_n(a)$ has weight enumerator $1 + y^{1+\wt(a)}$.
Hence, $B_n(a)$ and $B_n(b)$ are inequivalent if and only if
$\wt(a) \ne \wt(b)$.
The first two parts (i) and (ii) follow.

If $n$ is odd, then $B_n(\1_{n-1})^\perp$ is the unique
binary LCD $[n,n-1,2]$ code.
If $\wt(a)$ is even and $a \ne \1_{n-1}$, 
then $B_n(a)^\perp$ is a binary LCD $[n,n-1,1]$ code.
The last part (iii) follows from (ii).
\end{proof}

We give a complete classification of ternary LCD
$[n,k]$ codes for $k=1,n-1$.

\begin{prop}\label{prop:dim1T}
Let $\CC_{n,k}$ denote the set of all inequivalent
ternary LCD $[n,k]$ codes.
Let $\CC_{n,k,d}$ denote the set of all inequivalent
ternary LCD $[n,k,d]$ codes.
\begin{itemize}
\item[\rm (i)] 
$
|\CC_{n,1,d}|=
\begin{cases}
1 & \text{ if $d \not\equiv 0 \pmod 3$,} \\
0 & \text{ otherwise.}
\end{cases}
$\item[\rm (ii)] 
$
|\CC_{n,1}|=
|\CC_{n,n-1}|=
\begin{cases}
\frac{2n}{3}   & \text{ if $n \equiv 0 \pmod 3$,} \\
\frac{2n+1}{3} & \text{ if $n \equiv 1 \pmod 3$,} \\
\frac{2n+2}{3} & \text{ if $n \equiv 2 \pmod 3$.} \\
\end{cases}
$
\item[\rm (iii)] 
$
|\CC_{n,n-1,d}|=
\begin{cases}
\frac{2n}{3}   & \text{ if $n \equiv 0 \pmod 3$ and $d=1$,} \\
\frac{2n+1}{3}-1 & \text{ if $n \equiv 1 \pmod 3$ and $d=1$,} \\
1 & \text{ if $n \equiv 1 \pmod 3$ and $d=2$,} \\
\frac{2n+2}{3}-1 & \text{ if $n \equiv 2 \pmod 3$ and $d=1$,} \\
1 & \text{ if $n \equiv 2 \pmod 3$ and $d=2$,} \\
0 & \text{ otherwise.}
\end{cases}
$
\end{itemize}
\end{prop}
\begin{proof}
By considering equivalent codes,
we may assume without loss of generality that 
a ternary $[n,1]$  code has generator matrix 
of the form
$\left(
\begin{array}{cc}
1 & a
\end{array}
\right)
$,
where $a=(a_1,a_2,\ldots,a_{n-1}) \in \FF_3^{n-1}$.
We denote this code by $C_n(a)$.
It is trivial that
$C_n(a)$ is LCD if and only if $\wt(a) \not\equiv 2 \pmod 3$.
If $\wt(a) = \wt(b)$, then it is trivial that
$C_n(a)$ and $C_n(b)$ are equivalent.
If $\wt(a) \ne \wt(b)$, then
$C_n(a)$ and $C_n(b)$ are inequivalent 
since $C_n(a)$ has weight enumerator $1 + 2y^{1+\wt(a)}$.
Hence, $C_n(a)$ and $C_n(b)$ are inequivalent if and only if
$\wt(a) \ne \wt(b)$.
The first two parts (i) and (ii) follow.

If $n \not\equiv 0 \pmod 3$, then $C_n(\1_{n-1})^\perp$ is the unique
ternary LCD $[n,n-1,2]$ code.
If $\wt(a) \not\equiv 2 \pmod 3$ and $a \ne \1_{n-1}$, 
then $C_n(a)^\perp$ is a ternary LCD $[n,n-1,1]$ code.
The last part (iii) follows from (ii).
\end{proof}

\section{Classification of binary LCD codes of lengths up to 13}
\label{Sec:B}

In this section, we give a complete classification of
binary LCD codes of lengths up to $13$.
It is sufficient to consider $k \le n/2$,
since the dual code of a binary  LCD code is also LCD.

We describe how to complete a classification of
binary LCD codes of lengths up to $13$.
Every binary LCD $[n,k]$ code is equivalent to a binary
code with generator matrix of the form
$\left(
\begin{array}{cc}
I_k & A
\end{array}
\right)
$,
where $A$ is a $k \times (n-k)$ matrix.
The set of matrices $A$ was constructed, row by row.
Permuting the rows and columns of $A$ gives rise to different
generator matrices which generate equivalent binary codes.
Here, we consider some natural (lexicographical) order $<$ on
the set of vectors of length $n-k$.
We consider only matrices $A$, satisfying
the condition $r_1 \le r_2 \le \cdots \le r_k$,
where $r_i$ is the $i$-th row of $A$.
It is obvious that all binary codes, which must be checked to achieve
a complete classification, can be obtained.
By this method, we found all distinct binary LCD $[n,k]$ codes,
which must be checked to achieve
a complete classification for $2 \le k \le n/2 \le 11/2$.
By determining the equivalence or inequivalence 
for a given pair of binary codes, we obtained the set $\BB_{n,k}$
of all inequivalent binary LCD $[n,k]$ codes.
The mass formula~\eqref{eq:mf} shows that there is no
other binary LCD $[n,k]$ code.
This computation was performed in {\sc Magma}~\cite{Magma}. 
In principle, such a computation can be done by classifying 
binary LCD codes by the {\sc Magma} function {\tt IsIsomorphic}, 
then their automorphism groups can be calculated by 
{\tt AutomorphismGroup}.

For $n=12$ and $13$, the mass formula~\eqref{eq:mf} was
used to complete the classification, 
due to the computational complexity.
Let $\BB^*_{n,k}$ denote the set of all inequivalent
binary LCD $[n,k]$ codes $B$ with $d(B^\perp)=1$.
By Proposition~\ref{prop:dd1},
there is a set $\BB_{n-1,k}$ of  all inequivalent
binary LCD $[n-1,k]$ codes such that
$\BB^*_{n,k} =\{B^* \mid B \in \BB_{n-1,k}\}$,
where the construction of $B^*$ is listed in~\eqref{eq:extend}.
Note that $\BB^*_{n,n/2}$ is constructed by considering
the dual codes of the all inequivalent
binary LCD $[n-1,n/2-1]$ codes for the case $n=2k$.
Hence, it is sufficient to find 
all binary LCD $[n,k]$ codes $B$  with $d(B^\perp)\ge 2$,
which must be checked further for equivalences.
For these lengths,
the set of matrices $A$ was constructed, column by column.
Here, we consider some natural (lexicographical) order $<$ on
the set of non-zero vectors of length $k$.
We consider only matrices $A$, satisfying
the condition $c_1 \le c_2 \le \cdots \le c_{n-k}$,
where $c_i$ is the $i$-th column of $A$.
In this way, by adding new binary LCD codes,
we continued to construct the set $\BB'_{n,k}$
of inequivalent binary LCD $[n,k]$ codes until
$\sum_{B \in \mathcal{B}'_{n,k}} \frac{n!}{|\Aut(B)|}$
reaches the value $T_{2}(n,k)$.
When $\sum_{B \in \mathcal{B}'_{n,k}} \frac{n!}{|\Aut(B)|}$
reaches the value $T_{2}(n,k)$,
the classification was completed and $\mathcal{B}_{n,k}$ was obtained.

As a check, 
in order to verify that $\BB_{n,k}$ contains no pair of equivalent 
binary LCD $[n,k]$ codes and that
$\sum_{B \in \mathcal{B}_{n,k}} \frac{n!}{|\Aut(B)|}=T_{2}(n,k)$,
we employed the package GUAVA~\cite{guava} of GAP~\cite{gap}.
This calculation was done by using
the functions {\tt IsEquivalent} and {\tt AutomorphismGroup}.

In order to illustrate our approach, we consider the case 
$(n,k)=(6,3)$ as an example.
Let $B_{6,3,i}$ $(i=1,2,\ldots,8)$ be the binary LCD $[6,3]$ codes
with generator matrices 
$\left(
\begin{array}{cc}
I_3 & M_i
\end{array}
\right)$, 
where $M_i$ are listed in Table~\ref{Tab:C6}.
We verified that these binary codes are inequivalent.
Indeed, the binary codes $B_{6,3,i}$ have weight enumerators $W_{i}$,
where
\[
\begin{array}{ll}
W_1=1+3y^2+y^3+3y^5,&
W_2=1+3y^2+3y^3+y^5,\\
W_3=1+y^2+3y^3+2y^4+y^5,&
W_4=1+3y+3y^2+y^3,\\
W_5=1+y+3y^2+3y^3,&
W_6=1+y+y^2+y^3+2y^4+2y^5,\\
W_7=1+2y+y^2+y^3+2y^4+y^5,&
W_8=1+y+y^2+3y^3+2y^4.
\end{array}
\]
This shows also that these binary codes are inequivalent.
Since 
\[
|\Aut(B_{6,3,i})|
=36, 12,  4, 36, 12, 12, 12, 4\ (i=1,2,\ldots,8),
\]
we have
\[
\sum_{i=1}^8 \frac{6!}{|\Aut(B_{6,3,i})|} = 640
=T_2(6,3).
\]
The mass formula~\eqref{eq:mf}
shows that there is no other binary LCD $[6,3]$ code.

\begin{table}[thb]
\caption{LCD $[6,3]$ codes}
\label{Tab:C6}
\begin{center}
{\small
\begin{tabular}{c|c|c|c|c|c|c|c}
\noalign{\hrule height0.8pt}
$i$ & $M_i$ &$i$ & $M_i$ &
$i$ & $M_i$ &$i$ & $M_i$ \\
\hline
1 &
$\left(\begin{array}{c}
001\\
001\\
110
\end{array}\right)$
&
2&
$\left(\begin{array}{c}
001\\
001\\
011
\end{array}\right)$
&
3&
$\left(\begin{array}{c}
001\\
110\\
111
\end{array}\right)$
&
4 &
$\left(\begin{array}{c}
000\\
000\\
000
\end{array}\right)$
\\
5 &
$\left(\begin{array}{c}
001\\
001\\
000
\end{array}\right)$
&
6 &
$\left(\begin{array}{c}
001\\
111\\
000
\end{array}\right)$
&
7 &
$\left(\begin{array}{c}
110\\
000\\
000
\end{array}\right)$
&
8 &
$\left(\begin{array}{c}
001\\
011\\
000
\end{array}\right)$ \\
\noalign{\hrule height0.8pt}
\end{tabular}
}
\end{center}
\end{table}

For $2 \le k \le n/2$ and $n \le 13$,
we list in Table~\ref{Tab:C}
the numbers $N$ of all inequivalent binary LCD $[n,k]$ codes and
and the numbers $N_d$ of all inequivalent binary LCD $[n,k,d]$ codes.
We also list the numbers $N_{d^\perp}$ of the dual $[n,n-k,d^\perp]$
codes of the inequivalent binary LCD $[n,k]$ codes.
All  binary codes in the table
can be obtained electronically from
\url{http://www.math.is.tohoku.ac.jp/~mharada/LCD2/}.


\begin{table}[thb]
\caption{Smallest automorphism groups}
\label{Tab:AutB}
\begin{center}
{\small
\begin{tabular}{c|c|c|c|c|c|c|c}
\noalign{\hrule height0.8pt}
$(n,k)$ & $\Aut_s$ &$(n,k)$ & $\Aut_s$ &
$(n,k)$ & $\Aut_s$ &$(n,k)$ & $\Aut_s$ \\
\hline
$( 2,1)$& 1&$( 7,1)$& 144&$( 9,3)$&     8&$(11,3)$&   24\\
$( 3,1)$& 2&$( 7,2)$&  12&$( 9,4)$&     4&$(11,4)$&    4\\
$( 4,1)$& 6&$( 7,3)$&   4&$(10,1)$& 14400&$(11,5)$&    2\\
$( 4,2)$& 4&$( 8,1)$& 720&$(10,2)$&   288&$(12,2)$& 2880\\
$( 5,1)$&12&$( 8,2)$&  24&$(10,3)$&    16&$(12,3)$&   48\\
$( 5,2)$& 4&$( 8,3)$&   8&$(10,4)$&     4&$(12,4)$&    4\\
$( 6,1)$&36&$( 8,4)$&   4&$(10,5)$&     2&$(12,5)$&    2\\
$( 6,2)$& 8&$( 9,1)$&2880&$(11,1)$& 86400&$(12,6)$&    1\\
$( 6,3)$& 4&$( 9,2)$&  72&$(11,2)$&   864&&\\
\noalign{\hrule height0.8pt}
\end{tabular}
}
\end{center}
\end{table}

The smallest possible automorphism group of a binary LCD
code is the trivial group (of order $1$).
It is obvious that the unique binary LCD $[2,1]$ code has
trivial automorphism group.
We list in Table~\ref{Tab:AutB} the smallest value $\Aut_s$ 
among the orders of the automorphism groups of binary
LCD codes of lengths up to $12$.
From the table,  we have the following:

\begin{prop}
The smallest length $n >2$ for which there is a binary
LCD code of length $n$ with trivial automorphism group is $12$.
\end{prop}

We remark that
the total number of inequivalent binary LCD $[n,k]$
codes for $n=3,4,\ldots,12$ and $k \le n/2$
is $6897$, $14$ of which have
trivial automorphism groups.
The $14$ codes have parameters $[12,6,3]$.
As an example, we give a binary LCD $[12,6]$ code $B_{12}$ with
trivial automorphism group.
The code $B_{12}$ has generator matrix
$\left(
\begin{array}{cc}
I_6 & M_{12}
\end{array}
\right)$, 
where 
\[
M_{12}=
\left(\begin{array}{c}
101011\\
010110\\
110100\\
110001\\
001101\\
000011
\end{array}\right).
\]

Since the smallest possible automorphism group of a binary LCD
code is the group of order $1$,
there are at least $t_2(n,k)=\lceil T_2(n,k)/n! \rceil$
inequivalent binary LCD $[n,k]$ codes.
From
\begin{align*}
&
t_2(14,1)=1,
t_2(14,2)=1,
t_2(14,3)=1,
t_2(14,4)=18,
\\&
t_2(14,5)=574,
t_2(14,6)=4659,
t_2(14,7)=9282,
\end{align*}
we have 
\[
\sum_{k=1}^{13} t_2(14,k)=19790.
\]
Hence, there are at least $19791$ inequivalent binary LCD
codes of length $14$.
Of course, many binary LCD codes have substantial automorphism groups.
Thus, the above might be a poor lower bound.
Indeed, there are $30618$ inequivalent binary LCD codes of length $13$,
although we have
\[
\sum_{k=1}^{12} t_2(13,k)=2572.
\]

We continued the above process and completed a classification
of binary LCD codes for small dimensions.
More precisely, we give a classification of
binary LCD $[n,2]$ codes for $n \le 30$ and
binary LCD $[n,3]$ codes for $n \le 25$.
In order to save space,
we only list in Table~\ref{Tab:C2}
the numbers $N$ of all inequivalent binary LCD $[n,k]$ codes
for $k=2,3$ and $n \le 25,30$, respectively.

\begin{table}[thb]
\caption{Classification of binary LCD codes of dimensions $2$ and $3$}
\label{Tab:C2}
\begin{center}
{\small
\begin{tabular}{c|c|c|c|c|c|c|c}
\noalign{\hrule height0.8pt}
$(n,k)$ & $N$ &$(n,k)$ & $N$ &$(n,k)$ & $N$ &$(n,k)$ & $N$ \\
\hline
$(14,2)$&  66 &$(19,2)$& 136 &$(24,2)$& 270 &$(29,2)$& 431\\
$(15,2)$&  73 &$(20,2)$& 166 &$(25,2)$& 286 &$(30,2)$& 495\\
$(16,2)$&  93 &$(21,2)$& 178 &$(26,2)$& 335 &&\\
$(17,2)$& 101 &$(22,2)$& 214 &$(27,2)$& 354 &&\\
$(18,2)$& 126 &$(23,2)$& 228 &$(28,2)$& 410 &&\\
\hline
$(14,3)$& 380 &$(17,3)$&1120 &$(20,3)$&2648 &$(23,3)$&6074\\
$(15,3)$& 576 &$(18,3)$&1468 &$(21,3)$&3608 &$(24,3)$&7580\\
$(16,3)$& 772 &$(19,3)$&2058 &$(22,3)$&4568 &$(25,3)$&9870\\
\noalign{\hrule height0.8pt}
\end{tabular}
}
\end{center}
\end{table}

We end this section with giving the following remark.

\begin{rem}\label{rem}
It is a fundamental problem to determine the largest minimum weight
$d(n,k)$ among all binary LCD $[n,k]$ codes.
For $1 \le k \le n \le 12$,
the values $d(n,k)$ were determined in~\cite{bound}.
For $1 \le k \le n$ and $n=13,14,15,16$,
the values $d(n,k)$ were determined in~\cite{HS}.
Also, a classification of binary
LCD $[n,k]$ codes having the minimum weight $d(n,k)$
was given in~\cite{HS} for $1 \le k \le n \le 16$.
\end{rem}

\section{Classification of ternary LCD codes of lengths up to 10}
\label{Sec:T}

By an approach is similar to that used in the previous section, 
we completed a classification of ternary LCD $[n,k]$ codes 
for $2 \le k \le n/2 \le 10/2$.
For $1 \le n \le 7$, we found all distinct ternary LCD $[n,k]$ codes,
which must be checked to achieve a complete classification.
By determining the equivalence or inequivalence 
for a given pair of ternary codes, we obtained the set $\CC_{n,k}$.
The mass formula~\eqref{eq:mf3} shows that there is no
other ternary LCD $[n,k]$ code.
For $n=8,9$ and $10$, the mass formula~\eqref{eq:mf3} was
used to complete the classification, due to the computational complexity.
For ternary LCD codes, the following method was employed.
To test equivalence of ternary codes by a program in the language C, 
we used the algorithm given in~\cite[Section 7.3.3]{KO} as follows.
For a ternary $[n,k]$ code $C$, define 
the digraph $\Gamma(C)$ with vertex set 
$C \cup (\{1,2,\dots,n\}\times (\FF_3-\{0\}))$
and arc set 
$\{(c,(j,c_j))\mid c=(c_{1},\ldots,c_{n}) \in C,  1 \le j \le n\} 
\cup \{((j,y),(j,2y))\mid
1 \le j \le n,\ y \in \FF_3-\{0\}\}$.
Then, two ternary $[n,k]$ codes $C$ and $C'$ are equivalent
if and only if $\Gamma(C)$ and $\Gamma(C')$  are isomorphic.
We used {\sc nauty}~\cite{nauty}
for digraph isomorphism testing.
The automorphism group $\Aut(C)$ is isomorphic to the
automorphism group of $\Gamma(C)$.
This calculation was also done by using
{\sc nauty}~\cite{nauty}.

As a check, 
in order to verify that $\CC_{n,k}$ contains no pair of equivalent 
ternary LCD $[n,k]$ codes and that
$\sum_{C \in \mathcal{C}_{n,k}} \frac{2^nn!}{|\Aut(C)|}=T_{3}(n,k)$,
we employed the following method obtained by applying the method given
in~\cite[Section~2]{LTP}.
Let $C$ be a ternary $[n,k]$ code.
We expand each codeword of $C$ into a binary vector of length $2n$
by mapping the elements $0,1$ and $2$ of $\FF_3$
to the binary vectors $(0,0),(0,1)$ and $(1,0)$, respectively.
If there is a positive integer $t$ such that the codewords of
weight $t$ generate $C$,
then we have an $A_t \times 2n$ binary matrix $M(C)$
composed of the binary vectors obtained from the $A_t $ codewords of
weight $t$ in $C$.
If there is no positive integer $t$ such that the codewords of
weight $t$ generate $C$,
then by considering all codewords of $C$, we have a
$3^k \times 2n$ binary matrix $M(C)$.
Then, from $M(C)$,  we have an incidence structure 
$\cD(C)$ having $2n$ points.
This calculation was done by using
the {\sc Magma} function {\tt IncidenceStructure}.
If $C$ and $C'$ are equivalent ternary codes, then
$\cD(C)$ and $\cD(C')$ are isomorphic.
By the {\sc Magma} function {\tt IsIsomorphic}, 
we verified that all incidence structures $\cD(C)$ are non-isomorphic.
The automorphism group $\Aut(C)$ is isomorphic to
the stabilizer of $\{\{1,2\},\{3,4\},\ldots,\{2n-1,2n\}\}$
inside of the automorphism group of the incidence structure
$\cD(C)$.  
This calculation was done by using
the {\sc Magma} functions {\tt AutomorphismGroup}
and {\tt Stabilizer}.


We list in Table~\ref{Tab:T}
the numbers $N$ of the inequivalent ternary LCD $[n,k]$ codes 
and 
the numbers $N_d$ of the inequivalent ternary LCD $[n,k,d]$ codes 
for $2 \le k \le n/2$ and $n \le 10$.
We also list the numbers $N_{d^\perp}$ of the dual $[n,n-k,d^\perp]$
codes of the inequivalent ternary LCD $[n,k]$ codes.
All ternary codes in the table
can be obtained electronically from
\url{http://www.math.is.tohoku.ac.jp/~mharada/LCD3/}.


\begin{table}[thb]
\caption{Classification of ternary LCD codes}
\label{Tab:T}
\begin{center}
{\footnotesize
\begin{tabular}{c|c|ccccccc|cccc}
\noalign{\hrule height0.8pt}
$(n,k)$ & $N$ 
& $N_1$ & $N_2$ & $N_3$ & $N_4$ & $N_5$ & $N_6$ 
& $N_7$
& $N_{1^\perp}$ & $N_{2^\perp}$ & $N_{3^\perp}$ & $N_{4^\perp}$     \\
\hline
$(4, 2)$&   4&   2&  2&   &   &  &  &  &       2&  2&  &    \\
\hline
$(5, 2)$&   7&   3&  3&  1&   &  &  &  &       4&  3&  &    \\
\hline
$(6, 2)$&  11&   4&  4&  2&  1&  &  &  &       7&  4&  &    \\
$(6, 3)$&  17&   7&  8&  2&   &  &  &  &       7&  8& 2&    \\
\hline
$(7, 2)$&  16&   4&  6&  3&  3&  &  &  &      11&  5&  &    \\
$(7, 3)$&  36&  11& 17&  7&  1&  &  &  &      17& 17& 2&    \\
\hline
$(8, 2)$&  24&   5&  7&  4&  6& 2&  &  &      16&  8&  &    \\
$(8, 3)$&  74&  16& 31& 19&  8&  &  &  &      36& 37& 1&    \\
$(8, 4)$& 121&  36& 64& 19&  2&  &  &  &      36& 64&19& 2  \\
\hline
$(9, 2)$&  33&   6&  8&  5&  9& 4& 1&  &      24&  9&  &    \\
$(9, 3)$& 149&  24& 51& 40& 31& 3&  &  &      74& 74& 1&    \\
$(9, 4)$& 379&  74&178&105& 22&  &  &  &     121&218&40&    \\
\hline
$(10, 2)$&   45&    6&   10&   6&  11&  8& 3& 1  &      33&   12&    &    \\
$(10, 3)$&  290&   33&   80&  70&  84& 22& 1&    &     149&  140&   1&    \\
$(10, 4)$& 1293&  149&  458& 431& 249&  6&  &    &     379&  821&  93&    \\
$(10, 5)$& 2318&  379& 1209& 665&  65&   &  &    &     379& 1209& 665& 65 \\
\noalign{\hrule height0.8pt}
\end{tabular}
}
\end{center}
\end{table}

\begin{table}[thb]
\caption{Smallest automorphism groups}
\label{Tab:AutT}
\begin{center}
{\small
\begin{tabular}{c|c|c|c|c|c|c|c}
\noalign{\hrule height0.8pt}
$(n,k)$ & $\Aut_s$ &$(n,k)$ & $\Aut_s$ &
$(n,k)$ & $\Aut_s$ &$(n,k)$ & $\Aut_s$ \\
\hline
$(2, 1)$& 4 &$(5, 1)$& 96 &$(6, 3)$& 8    &$(8, 1)$& 11520 \\
$(3, 1)$& 8 &$(5, 2)$& 16 &$(7, 1)$& 1920 &$(8, 2)$& 96 \\
$(4, 1)$& 32&$(6, 1)$& 384&$(7, 2)$& 48   &$(8, 3)$& 8 \\
$(4, 2)$& 8 &$(6, 2)$& 24 &$(7, 3)$& 8    &$(8, 4)$& 2 \\
\noalign{\hrule height0.8pt}
\end{tabular}
}
\end{center}
\end{table}

The smallest possible automorphism group of a ternary LCD
code is the group of order $2$.
We list in Table~\ref{Tab:AutT} the smallest value $\Aut_s$ 
among the orders of the automorphism groups of ternary
LCD $[n,k]$ codes for $n \le 8$.
From the table,  we have the following:

\begin{prop}
The smallest length $n$ for which there is a ternary
LCD code of length $n$ with  automorphism group of order $2$ is 
$8$.
\end{prop}

We remark that
the total number of inequivalent ternary LCD $[n,k]$
codes for $n=2,3,\ldots,8$ and $k \le n/2$
is $336$, one of which has
automorphism group of order $2$.
The code has parameters $[8,4,3]$.
As an example, we give a ternary LCD $[8,4]$ code $C_{8}$ with
automorphism group of order $2$.
The code $C_{8}$ has generator matrix
$\left(
\begin{array}{cc}
I_4 & M_{8}
\end{array}
\right)$, 
where 
\[
M_{8}=
\left(\begin{array}{c}
2001\\
2212\\
1100\\
1012
\end{array}\right).
\]

Since the smallest possible automorphism group of a ternary LCD
code is the group of order $2$,
there are at least $t_3(n,k)=\lceil T_3(n,k)/(2^{n-1}n!) \rceil$
inequivalent ternary LCD $[n,k]$ codes.
From
\begin{align*}
&
t_3(11,1)=1,
t_3(11,2)=1,
t_3(11,3)=4,
\\&
t_3(11,4)=319,
t_3(11,5)=2869,
\end{align*}
we have 
\[
\sum_{k=1}^{10} t_3(11,k)=6388.
\]
Hence, there are at least $6389$ inequivalent ternary LCD
codes of length $11$.
Of course, many ternary LCD codes have substantial automorphism groups.
Thus, the above might be a poor lower bound.
Indeed, there are $5588$ inequivalent ternary LCD codes of length $10$,
although we have
\[
\sum_{k=1}^{9} t_3(10,k)=447.
\]

Similar to Remark~\ref{rem}, 
we end this section with giving the following remark.

\begin{rem}
Let $d_3(n,k)$ and $d^{\text{all}}_3(n,k)$ denote
the largest minimum weight
among all ternary LCD $[n,k]$ codes and
among all ternary $[n,k]$ codes, respectively.
Suppose that $2 \le k \le n-2$ and $4 \le n \le 10$.
From Table~\ref{Tab:T}, we have
\[
d_3(n,k)=
\begin{cases}
d^{\text{all}}_3(n,k)-1 &\text{ if } (n,k) \in S,\\
d^{\text{all}}_3(n,k) & \text{ otherwise,} 
\end{cases}
\]
where
\[
S=\{(4,2),(7,2),(8,2),(8,3),(9,3),(9,4),(9,5),(10,4),(10,5),(10,6)\}.
\]
Note that the values $d^{\text{all}}_3(n,k)$  can be found in~\cite{G}.
\end{rem}

\bigskip
\noindent
{\bf Acknowledgment.}
This work was supported by JSPS KAKENHI Grant Number 15H03633.



\begin{landscape}

\begin{table}[thb]
\caption{Classification of binary LCD codes}
\label{Tab:C}
\begin{center}
{\footnotesize
\begin{tabular}{c|c|cccccccc|ccccc}
\noalign{\hrule height0.8pt}
$(n,k)$ & $N$ 
& $N_1$ & $N_2$ & $N_3$ & $N_4$ & $N_5$ & $N_6$ 
& $N_7$ & $N_8$ 
& $N_{1^\perp}$ & $N_{2^\perp}$ & $N_{3^\perp}$ & $N_{4^\perp}$ \\
\hline
$( 4,2)$ &    4  & 2 & 2 &  &   &  &   &  &        & 2 & 2 &  &    \\
\hline
$( 5,2)$ &    5  & 2 & 3 &  &   &  &   &  &         & 4 & 1 &  &    \\
\hline
$( 6,2)$ &    9  & 3 & 4 & 2 &  &   &  &   &       & 5 & 4 &  &    \\
$( 6,3)$ &    8  & 5 & 3 &  &   &  &   &  &        & 5 & 3 &  &    \\
\hline
$( 7,2)$ &   11  & 3 & 5 & 2 & 1 &  &   &  &       & 9 & 2 &  &    \\
$( 7,3)$ &   17  & 9 & 7 & 1 &  &   &  &   &       & 8 & 9 &  &    \\
\hline
$( 8,2)$ &   17  & 4 & 6 & 4 & 2 & 1 &  &   &      & 11 & 6 &  &    \\
$( 8,3)$ &   26  & 11 & 12 & 3 &  &   &  &   &     & 17 & 9 &  &    \\
$( 8,4)$ &   42  & 17 & 24 & 1 &  &   &  &   &     & 17 & 24 & 1 &    \\
\hline
$( 9,2)$ &   20  & 4 & 7 & 4 & 3 & 1 & 1 &  &       & 17 & 3 &  &     \\
$( 9,3)$ &   49  & 17 & 20 & 11 & 1 &  &   &  &     & 26 & 23 &  &    \\
$( 9,4)$ &   81  & 26 & 49 & 5 & 1 &  &   &  &      & 42 & 37 & 2 &    \\
\hline
$(10,2)$ &   29  & 5 & 8 & 6 & 4 & 4 & 2 &  &       & 20 & 9 &  &    \\
$(10,3)$ &   72  & 20 & 29 & 18 & 4 & 1 &  &   &    & 49 & 23 &  &    \\
$(10,4)$ &  186  & 49 & 109 & 23 & 5 &  &   &  &    & 81 & 103 & 2 &  \\
$(10,5)$ &  204  & 81 & 112 & 11 &  &   &  &   &    & 81 & 112 & 11 & \\
\hline
$(11,2)$ &   33  & 5 & 9 & 6 & 5 & 4 & 4 &  &   & 29 & 4 &  &   \\
$(11,3)$ &  123  & 29 & 42 & 35 & 11 & 6 &  &   &  & 72 & 51 &  &    \\
$(11,4)$ &  348  & 72 & 195 & 61 & 20 &  &   &  &   & 186 & 161 & 1 &    \\
$(11,5)$ &  606  & 186 & 350 & 66 & 4 &  &   &  &   & 204 & 386 & 15 & 1  \\
\hline
$(12,2)$ &   45  & 6 & 10 & 8 & 6 & 7 & 6 & 2 &    & 33 & 12 &  &     \\
$(12,3)$ &  174  & 33 & 56 & 48 & 22 & 14 & 1 &   && 123 & 51 &  &     \\
$(12,4)$ &  744  & 123 & 369 & 170 & 76 & 6 &  &   && 348 & 396 &  &     \\
$(12,5)$ & 1584  & 348 & 909 & 290 & 37 &  &   &  && 606 & 956 & 22 &     \\
$(12,6)$ & 2426  & 606 & 1622 & 187 & 11 &  &   &  && 606 & 1622 & 187 & 11  \\
\hline
$(13, 2)$&   50&     6&   11&    8&   7&  7& 8& 2& 1&   45&    5&    &  \\
$(13, 3)$&  277&    45&   75&   77&  39& 35& 6&  &  &  174&  103&    &  \\
$(13, 4)$& 1363&   174&  598&  341& 217& 31& 2&  &  &  744&  619&    &  \\
$(13, 5)$& 4576&   744& 2354& 1178& 295&  5&  &  &  & 1584& 2965&  27&  \\
$(13, 6)$& 9036&  1584& 5900& 1406& 146&   &  &  &  & 2426& 6086& 520& 4\\
\noalign{\hrule height0.8pt}
\end{tabular}
}
\end{center}
\end{table}

\end{landscape}


\begin{thebibliography}{99}

\bibitem{guava}
R. Baart, 
T. Boothby, 
J. Cramwinckel, 
J. Fields, 
D. Joyner, 
R. Miller, 
E. Minkes, 
E. Roijackers, 
L. Ruscio and
C. Tjhai,
GAP package GUAVA, Version 3.1.3; 2016,
Available online at \url{http://www.gap-system.org/Packages/guava.html}.

\bibitem{Magma}W. Bosma, J. Cannon and C. Playoust,
The Magma algebra system I: The user language,
{\sl J. Symbolic Comput.}
{\bf 24} (1997), 235--265.



\bibitem{mf}
C. Carlet, S. Mesnager, C. Tang and Y. Qi,
New characterization and parametrization of LCD codes,
arXiv:1709.03217.

\bibitem{CMTQ2}
C. Carlet, S. Mesnager, C. Tang, Y. Qi and R. Pellikaan,
Linear codes over $\FF_q$ are equivalent to LCD codes for $q >3$,
{\sl IEEE\ Trans.\ Inform.\ Theory}
{\bf 64}  (2018),  3010--3017.


\bibitem{bound}
L. Galvez, J.-L. Kim, N. Lee, Y.G. Roe and B.-S. Won,
Some bounds on binary LCD codes,
{\sl Cryptogr.\ Commun.}
{\bf 10} (2018), 719--728.

\bibitem{gap}
The GAP Group, 
GAP--Groups, Algorithms, and Programming, Version 4.8.10; 2018,
Available online at \url{http://www.gap-system.org}.

\bibitem{G} M. Grassl,
Code tables: Bounds on the parameters of various types of codes,
Available online at 
\url{http://www.codetables.de/},
Accessed on 2018-03-30.


\bibitem{HS} M. Harada and K. Saito,
Binary linear complementary dual codes,
arXiv: 1802.06985.


\bibitem{KO}P. Kaski and  P.R.J. \"Osterg\aa rd,
{\sl Classification Algorithms for Codes and Designs}, 
Springer, Berlin, 2006.

\bibitem{LTP} C.W.H. Lam, L. Thiel and A. Pautasso,
On ternary codes generated by Hadamard matrices of order $24$,
{\sl Congr.\ Numer.}
{\bf 89} (1992), 7--14. 


\bibitem{Massey}J.L. Massey, 
Linear codes with complementary duals,
{\sl Discrete Math.}
{\bf 106/107} (1992), 337--342.


\bibitem{nauty}B.D. McKay and A. Piperno, 
nauty and Traces User's Guide (Version 2.6),
Available online at
\url{http://users.cecs.anu.edu.au/~bdm/nauty/nug26.pdf}.




\end{thebibliography}
\end{document}